\documentclass{llncs}
\usepackage{amssymb}
\usepackage{amsmath}
\usepackage{graphicx}
\usepackage[]{algorithm2e}

\newtheorem{lemmaa}{Assertion}

\newcommand{\F}{\mathbb{F}}

\newcommand{\Z}{\mathbb{Z}}

\def\bbbz{{\mathchoice {\hbox{$\sf\textstyle Z\kern-0.4em Z$}}
 {\hbox{$\sf\textstyle Z\kern-0.4em Z$}} {\hbox{$\sf\scriptstyle
 Z\kern-0.3em Z$}} {\hbox{$\sf\scriptscriptstyle Z\kern-0.2em
 Z$}}}}

\newcommand{\card}[1]{\left|#1\right|}

\newcommand{\tr}{\mbox{\rm Tr}}

\newenvironment{tabcen}[1]{\begin{center}\begin{tabular}{#1}}{\end{tabular}\end{center}}

\newcommand{\proto}[3]{\begin{tabcen}{p{.4\textwidth}cp{.4\textwidth}}
\hspace{1.8cm}{#1} & & \hspace{1.8cm}{#2} \\ \hline
#3
\hline\end{tabcen}}

\newcommand{\grn}[1]{\mbox{\rm GR}\left(#1\right)}

\title{Proof of Correspondence between Keys and Encoding Maps in an Authentication Code}
\titlerunning{Proof of Correspondence}  %

\author{Juan Carlos Ku-Cauich,  Guillermo Morales-Luna\inst{1} \and Horacio Tapia-Recillas\inst{2}
}
\authorrunning{Ku-Cauich et al.}   
\tocauthor{J. C. Ku-Cauich (CINVESTAV-IPN),
G. Morales-Luna (CINVESTAV-IPN),
H. Tapia-Recillas  (UAM-I),
}
\institute{Computer Science, CINVESTAV-IPN, Mexico City, Mexico, 
\email{jckc35@hotmail.com,gmorales@cs.cinvestav.mx}
\and
Mathematics Department, Universidad Aut\'onoma Metropolitana-I, Mexico City, Mexico, 
\email{htr@xanum.uam.mx}
}

\begin{document}
\maketitle

\begin{abstract}
In a former paper the authors  introduced two new systematic authentication codes based on the Gray map over a Galois ring. In this paper, it is proved the one-to-one onto correspondence between keys and encoding maps for the second introduced authentication code.

 \keywords{Authentication schemes, resilient maps, Gray map.}

{\bf 2010 MSC} Primary: 11T71; Secondary: 14G50, 94A60, 94A62.
\end{abstract}

\section{Introduction}

Systematic authentication codes without secrecy were defined in~\cite{DingN04}.
In~\cite{ku15an} two new systematic authentication codes based on the Gray map on a Galois ring are introduced with the purpose of optimally reducing the impersonation and substitution probabilities. The first code is another example of a previously constructed code using the Gray map on Galois rings and modules over these rings~\cite{Ku-CauichT13,OzbudakS06}. 
The second code generalises the construction in~\cite{Ku-CauichT13}, on the assumption of the existence of an appropriate class of bent functions.
For the first code, the existence of the bijection between the key space and the set of encoding maps is proved in this paper in a rather long but exhaustive way.

\section{Refreshment of basic notions}

\subsection{General  systematic authentication codes} \label{sc.sac01}

We recall that a {\em systematic authentication code without secrecy}~\cite{DingN04} is a structure $(S,T,K,E)$ where $S$ is the {\em source state space}, $T$ is the {\em tag space}, $K$ is the {\em key space} and $E=\left(e_k\right)_{k\in K}$ is a sequence of {\em encoding rules} $S\to T$.

A {\em transmitter} and a {\em receiver} agree to a secret key $k\in K$. Whenever a source $s\in S$ must be sent, the participants proceed according to the following protocol:

\proto{Transmitter}{Receiver}{
evaluates $t = e_k(s)\in T$ &  &  \\
forms the pair $m=(s,t)$ & $\stackrel{m}{\longrightarrow}$ & receives $m'=(s',t')$,  \\
 &  & evaluates $t'' = e_k(s')\in T$ \\
 &  & if $t'=t''$ then accepts $s'$, otherwise the message $m'$ is rejected \\
}

The communicating channel is public, thus it can be eavesdropped upon by an {\em intruder} able to perform either {\em impersonation} or {\em substitution} attacks through the public channel. The intruder's success probabilities for impersonation and substitution are, respectively~\cite{Stinson92}
\begin{eqnarray}
p_I &=& \max_{(s,t)\in S\times T} \frac{{\card{\{k\in K|\ e_k(s)=t\}}}}{\card{K}}  \label{eq.d01} \\
p_S &=& \max_{(s,t)\in S\times T}\max_{(s',t')\in (S-\{s\})\times T} \frac{\card{\{k\in K|\ e_k(s)=t\ \&\ e_k(s')=t'\}}}{\card{\{k\in K|\ e_k(s)=t\}}}  \label{eq.d02}
\end{eqnarray}

\subsection{The first systematic authentication code} \label{esquem2}

The first systematic authentication code introduced in~\cite{ku15an} is constructed as follows:

Let $p$ be a prime number, $r,\ell,n\in\Z^+$ and $q = p^{\ell}$. Let $A=\grn{p^r,\ell}$ and $B=\grn{p^r,\ell n}$ be the corresponding Galois rings of degrees $\ell$ and $\ell n$. 
We denote by $T(A) =\{0\}\cup\left(\xi_A^j\right)_{j=0}^{q -2}$ the set of Teichm\"uller representatives of $\F_q$ in $A$. Then $p^{r-1}A = \{a\,p^{r-1}|\ a\in T(A)\}$. We define $\Xi=(0,\rho(\xi_A),\ldots,\rho(\xi_A^{q-2}),\rho(\xi_A^{q-1}))\in \F_q^{q}$
and
$L=\{r_{0}+r_{1}p+\cdots + r_{r-2}p^{r-2}~|~ r_{0},\ldots,
r_{r-2}\in{ T}(A) \}\subset A\backslash p^{r-1}A\cup \{0\}.$ 
Since $\left\langle p^{r-1}\right\rangle=\{ap^{r-1}~|~ a\in {T}(A)\}$, if $a,b \in L$ then $a-b \in A\backslash p^{r-1}A$. 

Similarly, $T(B)$ is the set of the Teichm\"uller representatives of $\ F_{q^{m}}$ in $B$.

Let $n\in\Z^+$ and $t\leq n$. For any $i<n$, we denote $e_i=\left(\delta_{ij}\right)_{j=0}^{n-1}$ as the $i$-th ``canonical'' vector.
For any $b\in T(B)^n$, let
\begin{eqnarray}
 X_{b,t} &=& \{\sum_{j=0}^{t-2}b_je_j,b_{t-1}e_{t-1},\ldots,b_{n-1}e_{n-1}\}\subset B^n, \nonumber \\
 N &=& \bigcup_{b\in T(B)^n}X_{b,t}, \label{eq.sacn} \\
 L &=& \left\{\sum_{i=0}^{r-2} r_ip^i|\ (r_0,\ldots,r_{r-2})\in T(A)^{r-1}\right\}. \label{eq.sacl}
\end{eqnarray}
Then 
$\card{X_{b,t}} = n-t+1$, $\card{N} = q^{m(t-1)}+(n-(t-1))q^m$, $\card{L} = q^{r-1}$, \linebreak $L \subset (A-p^{r-1}A)\cup\{0\}$ and also $\forall u,v\in L:\ (u-v)\in (A-p^{r-1}A)\cup\{0\}.$ 
Let us consider an $(r-1) n$-subset of $T(A)-\{0,1\}$, 
\begin{equation}
\eta = \left\{\eta_k\right\}_{k=0}^{(r-1)n-1}, \label{eq.eta}
\end{equation}
and 
\begin{equation}
D_{\eta} = \left\{(\eta_{(i-1)n+j},p^ie_j)|\ 1\leq i\leq r-1\,,\,0\leq j\leq n-1\right\}. \label{eq.deta}
\end{equation}
Then $D_{\eta}\subset A\times B^n$ and $\card{D_{\eta}} = (r-1)n$.

Let us write $T(B)=\{0\}\cup\left(\xi_B^k\right)_{k=0}^{q^m -2}$,  
$G(T(B)) = \{\xi_B^k|\ \mbox{gcd}(k,q^m -1)=1\}$ and
 $\theta = \left\{\theta_j\right\}_{j=0}^{n-1}$, which is an $n$-sequence of $G(T(B))$ (repetitions are allowed), and $\zeta\in T(B)-\{0\}$.
For each integer $k$, with $0\leq k\leq q^m -(r-1)n-2$, let
$$T_{\theta\zeta k} =  \left\{(\theta_j^i,(\zeta+\theta_j^i\, p^{1+(k\bmod(r-1))})e_j)|\ 0\leq i\leq q^m -2\,,\,0\leq j\leq n-1\right\}.$$
Then $T_{\theta\zeta k}\subset B\times B^n$ and $\card{T_{\theta\zeta k}} = (q^m -1)n$.
Now, let $Z = \left\{\zeta_k\right\}_{k=0}^{q^m -(r-1)n-2}$ be a subset of $T(B)-\{0\}$, with $(q^m -1-(r-1)n-1)$ elements, such that $Z \cap\eta = \emptyset$, and 
\begin{equation}
{\bf T}_{\eta\theta Z} =  D_{\eta}\cup\bigcup_{k=0}^{q^m -(r-1)n-2} T_{\theta\zeta_k k}. \label{eq.teta}
\end{equation}
Then ${\bf T}_{\eta\theta Z}\subset B\times B^n$ and
\begin{eqnarray*}
\card{{\bf T}_{\eta\theta Z}} &=& (r-1)n + (q^m -1-(r-1)n)(q^m -1)n \\
 &=& \left[(r-1) + \left[(q^m  -1)- (r-1)n\right](q^m -1)\right] n
\end{eqnarray*}

Given a $t$-resilient map~\cite{ku15an} $f:B^n\to B$, for each $s=(s_0,s_1,s_2)\in S $ and each $w\in p^{r-1}A$, consider the map
\begin{eqnarray}
v_{s,w}:B^n &\to& A \nonumber \\
x &\mapsto&  \begin{array}[t]{rcl}
v_{s,w}(x) &=& \tr_{B/A}(s_0\,f(x) + s_1\cdot x) + s_2 + w \\
 &=& \gamma_{s_0s_1f}(x) + s_2 + w
\end{array}  \label{eq.rs051}
\end{eqnarray}
Let
\begin{eqnarray}
u_{s,w} &=& \left(\Phi\left(v_{s,w}(x)\right)\right)_{x\in B^n} \in \left(\F_q^{q^{r-1}}\right)^{q^{rmn}} ,   \nonumber \\
u_s &=& \left(u_{s,w}\right)_{w\in p^{r-1}A} \in \left(\F_q^{q^{r-1}}\right)^{q^{rmn+1}} .   \label{eq.rs052} 
\end{eqnarray}
Since $\card{p^{r-1}A} = q$, we have $\left(\F_q^{q^{r-1}}\right)^{q^{rmn+1}} \simeq \F_q^{q^{r(mn+1)}}$, thus we may assume $u_s\in\F_q^{q^{r(mn+1)}}$.

This paper is devoted to prove the following theorem which is the main contribution in~\cite{ku15an}:

\begin{theorem}\label{pr.04} The map $K\to E$, $k\mapsto e_k$, is one-to-one.
\end{theorem}

\begin{proof}
The theorem is clearly equivalent to the following statement:
\begin{equation}
\forall k_0,k_1\in K:\ \left[k_0\not=k_1 \ \Longrightarrow\ \exists s\in S:\ \pi_{k_0}(u_s)\not= \pi_{k_1}(u_s)\right] \label{eq.pp01}
\end{equation}
where $u_s$ is given by relation~(\ref{eq.rs052}), and
$\pi_k(u_s)$ is the $k$-th entry of the element $u_s.$

According to~(\ref{eq.rs052}), each element $u_s$, $s\in S$, is the concatenation of $q$ arrays $u_{s,w}$, each of length $q^{rmn}$. The index range $\{0,\ldots,q^{r(mn+1)}-1\}$ of the element $u_s$ can be split as the concatenation of $q^{rmn+1}$ integer intervals
$$K_{x,w} = \{\mbox{indexes of entries with the value }\Phi\left(v_{s,w}(x)\right)\}$$
with $(x,w)\in B^n\times p^{r-1}A$, and each integer interval $K_{x,w}$ has length $q^{r-1}$. 

We recall at this point that
$\card{B^n\times p^{r-1}A} = q^{rmn} q = q^{rmn+1}.$ 
Let $\alpha_b:B^n\to\{0,\ldots,q^{rmn}-1\}$, $\alpha_a:p^{r-1}A\to\{0,\ldots,q-1\}$ be the corresponding natural bijections. Then we may identify
$$K_{x,w}\approx\{k\in K|\ k_{x,w}q^{r-1}\leq k\leq k_{x,w}q^{r-1}+(q^{r-1}-1)\},$$
where
\begin{equation}
\forall (x,w)\in B^n\times p^{r-1}A:\ \ \ \ k_{x,w} = \alpha_b(x) q + \alpha_a(w). \label{eq.pq01}
\end{equation}

Let $k_0,k_1\in K\approx\{0,\ldots,q^{r(mn+1)}-1\}$ be two keys such that $k_0\not=k_1$. Depending on the intervals $K_{x,w}$ in which these keys fall, we may consider four mutually disjoint and exhaustive cases.
\begin{itemize}
\item {\em Case I:} \ \ $\exists w\in p^{r-1}A,\exists x\in B^n$: $k_0\in K_{x,w}$ \& $k_1\in K_{x,w}$.
\item {\em Case II:} \ $\exists w\in p^{r-1}A,\exists x,y\in B^n$: $x\not= y$ \& $k_0\in K_{x,w}$ \& $k_1\in K_{y,w}$.
\item {\em Case III:} $\exists w_0,w_1\in p^{r-1}A,\exists x\in B^n$: $w_0\not= w_1$ \& $k_0\in K_{x,w_0}$ \& $k_1\in K_{x,w_1}$.
\item {\em Case IV:} $\exists w_0,w_1\in p^{r-1}A,\exists x,y\in B^n$:
$$w_0\not= w_1\ \&\ x\not= y\ \&\ k_0\in K_{x,w_0}\ \&\ k_1\in K_{y,w_1}.$$
\end{itemize}
The analysis of these cases, giving a full proof of the theorem, is rather extensive and it is provided in the following section. \end{proof}

\section{Proof of Proposition 3 in~\cite{ku15an}}

The detailed proof of Theorem~\ref{pr.04} is presented in this section. 
The plan of the proof is sketched as Plan~\ref{tb.pl}. In what follows, we will list extensively all the assertions claimed in the proof plans.

\renewcommand{\tablename}{Plan}
\begin{table}
\fbox{ \IncMargin{1em}
\begin{algorithm}[H]
 \eIf{Case I holds}
  {{\bf I.} See Assertion~\ref{lm.I}}
 {\eIf{Case II holds}
  {let $k_{00} = k_0-k_{x,w}$ and $k_{10} = k_1-k_{y,w}$ \;
   \eIf{$k_{00} = k_{10}$}
  {proceed as in Plan~\ref{tb.pl0}} 
  {proceed as in Plan~\ref{tb.pl1}}
  }
 {\eIf{Case III holds}
  {let $k_{00} = k_0 - k_{x,w_0}$ and $k_{10} = k_1 - k_{x,w_1}$, according to~(\ref{eq.pq01}) \;
  \eIf{$k_{00} = k_{10}$}
 {{\bf III.0} See Assertion~\ref{lm.III0}}
 {pick  $(s_0,s_1)\in\{0\}\times\left(N-\{0\}\right)$ arbitrarily \;
 \eIf{$\pi_{k_{00}}\circ \Phi\left(\tr_{B/A}(s_0\,f(x) + s_1\cdot x) + w_0\right)=\pi_{k_{10}}\circ \Phi\left(\tr_{B/A}(s_0\,f(x) + s_1\cdot x) + w_1\right)$}
 {{\bf III.1.0} See Assertion~\ref{lm.III10}}
 {{\bf III.1.1} See Assertion~\ref{lm.III11}}}  }
  {(at this point, {\em Case IV} necessarily does hold ) \\
  let $k_{00} = k_0 - k_{x,w_0}$ and $k_{10} = k_1 - k_{x,w_1}$, according to~(\ref{eq.pq01}) \;
 \eIf{$\pi_{k_{00}}\circ \Phi\left(\tr_{B/A}(f(x) )\right) = \pi_{k_{10}}\circ \Phi\left(\tr_{B/A}(f(y) )\right)$}
 {{\bf IV.0} See Assertion~\ref{lm.IV0}}
 {{\bf IV.1} See Assertion~\ref{lm.IV1}}
  }}}
\end{algorithm}}
\caption{Plan of the proof of Theorem~\ref{pr.04}. \label{tb.pl} }
\end{table}

\begin{table}
\fbox{ \IncMargin{1em}
\begin{algorithm}[H]
 {choose $j\in\{0,\ldots,n-1\}$ such that the $j$-th entry of $x-y$ is not zero, namely $x_j-y_j\not=0$ \;
  \eIf{$x_j-y_j\in p^{r-1}B-\{0\}$}
 {{\bf II.0.0} See Assertion~\ref{lm.II00}}
 {there are $\theta\in T_B-T_A$, and $t\leq r-1$ such that $\tr_{B/A}(\theta p^t(x_j-y_j))\in p^{r-1}A-\{0\}$ \;
\eIf{$\tr_{B/A}(x_j) = \tr_{B/A}(y_j)$}
 {let $\zeta\in T_A-\{0\}$ be such that $(\theta,(\zeta+\theta p^t)e_j)\in\bigcup_{k=0}^{q^m -(r-1)n-2} T_{\theta\zeta_k k}$ as defined at~(\ref{eq.teta}) \;
we have 
\begin{eqnarray*}
 \tr_{B/A}(\zeta  x_j) &=& \sum_{k=0}^{r-1} d_kp^k = \tr_{B/A}(\zeta  y_j) \\
 \tr_{B/A}(\theta p^t x_j) &=& \sum_{k=0}^{r-2} a_kp^k + a_{r-1}p^{r-1} \\  
 \tr_{B/A}(\theta p^t y_j) &=& \sum_{k=0}^{r-2} a_kp^k + b_{r-1}p^{r-1}
\end{eqnarray*}
with $a_{r-1}\not=b_{r-1}$ \;
let $(s_0,s_1) = (\theta, (\zeta+\theta p^t)e_j)$ \;
\eIf{$\pi_{k_{00}}\circ \Phi\left(\tr_{B/A}(\theta f(x))\right)=\pi_{k_{10}}\circ \Phi\left(\tr_{B/A}(\theta f(y))\right)$}
 {{\bf II.0.1.0.0} See Assertion~\ref{lm.II0100}}
 {{\bf II.0.1.0.1} See Assertion~\ref{lm.II0101}}}
 {{\bf II.0.1.1 } \eIf{$\pi_{k_{00}}\circ \Phi\left(\tr_{B/A}(f(x))\right) = \pi_{k_{10}}\circ \Phi\left(\tr_{B/A}(f(y))\right)$}
 {There is a $t$, $0\leq t\leq r-1$, such that $\tr_{B/A}(p^t(x_j-y_j))\in p^{r-1}A-\{0\}$ (here the hypothesis $\tr_{B/A}(x_j) \not= \tr_{B/A}(y_j)$ is very important) \;
 \eIf{$t=0$}
 {{\bf II.0.1.1.0.0} See Assertion~\ref{lm.II01100}}
 {{\bf II.0.1.1.0.1} See Assertion~\ref{lm.II01101}}}
 {{\bf II.0.1.1.1} See Assertion~\ref{lm.II0111}}}  }  }
\end{algorithm}}
\caption{First branch of Case II. \label{tb.pl0} }
\end{table}

\begin{table}
\fbox{ \IncMargin{1em}
\begin{algorithm}[H]
let $j\in\{0,\ldots,n-1\}$ be such that $x_j-y_j\not=0$ \;
\eIf{$x_j-y_j\in p^{r-1}B-\{0\}$}
 {{\bf II.1.0} See Assertion~\ref{lm.II10}}
 {there exist $\theta\in (T_B-T_A)\cup\{1\}$ and $t\in\{1,\cdots,r-1\}$ such that $\tr_{B/A}(\theta p^t(x_j-y_j))\in p^{r-1}A-\{0\}$ \;
\eIf{$\tr_{B/A}(x_j)=\tr_{B/A}(y_j)$}
 {let $\zeta\in T_A-\{0\}$ be such that the pair $(s_0,s_1) = (\theta,(\zeta+\theta p^t)e_j)$  is included in the set $\bigcup_{k=0}^{q^m -(r-1)n-2} T_{\theta\zeta_k k}$ as defined at~(\ref{eq.teta}) \;
\eIf{$\Phi\left(\tr_{B/A}(\theta f(x))\right) = \Phi\left(\tr_{B/A}(\theta f(y))\right)$}
 {{\bf II.1.1.0.0} See Assertion~\ref{lm.II1100}}
 {{\bf II.1.1.0.1} See Assertion~\ref{lm.II1101}} }
 {proceed as in statement {\bf II.0.1.1} of Plan~\ref{tb.pl0}} }
\end{algorithm}}
\caption{Second branch of Case II. \label{tb.pl1} }
\end{table}
\renewcommand{\tablename}{Table}

\begin{lemmaa} \label{lm.I}
 Upon the condition underlying the statement {\bf I} in Plan~\ref{tb.pl}, the claim~(\ref{eq.pp01}) holds.
\end{lemmaa}

\begin{proof}
Let $(s_0,s_1)\in\{0\}\times\left(N-\{0\}\right)$ and
$$\tr_{B/A}(s_0 f(x) + s_1\cdot x) = \sum_{i=0}^{r-2} a_ip^i + a_{r-1}p^{r-1}.$$
For each $k\in\{0,\ldots,r-2\}$, there exists $y^{(k)}= \sum_{i=0}^{r-2} y_{ik}p^i\in L$ such that
$$\tr_{B/A}(s_0 f(x) + s_1\cdot x)+y^{(k)} = \left\{\begin{array}{ll}
 a_kp^k + a_{r-1}p^{r-1} & \mbox{if } a_k\not =0 \\
 y_{kk}p^k + a_{r-1}p^{r-1} & \mbox{if } a_k =0\ \&\ y_{kk}\not =0
\end{array}\right.$$
Thus,
\begin{eqnarray*}
 & & \ \ \ \Phi\left(\tr_{B/A}(s_0 f(x) + s_1\cdot x) + y^{(k)} + w\right) \\ \rule{0ex}{.7cm}
 &=& \left\{\begin{array}{rl}
 \Phi\left(a_kp^k \right) + \Phi\left(a_{r-1}p^{r-1} + w\right) & \mbox{if } a_k\not =0 \\
 \Phi\left(y_{kk}p^k  \right)+ \Phi\left(a_{r-1}p^{r-1} + w\right) & \mbox{if } a_k =0\ \&\ y_{kk}\not =0
\end{array}\right.
\end{eqnarray*}
We have that $(s_0,s_1,y^{(k)})\in S$ and $w\in p^{r-1}A$.

Now, let $k_{00} = k_0-k_{x,w}$ and $k_{10} = k_1-k_{x,w}$. Let us consider the following possibilities:
\begin{itemize}
\item $q\not|(k_{10}-k_{00})$: By taking $a_{r-2} \not=0$, all other coefficients zero, and $s=(s_0,s_1,s_2)$, the $k_{00}$-projection of $u_{s,w}$ (see~(\ref{eq.rs052})) differs from its $k_{10}$-projection, thus $\pi_{k_0}(u_s)\not= \pi_{k_1}(u_s)$.
\item $q|(k_{10}-k_{00})$ and ($\exists d$: $1\leq d\leq r-1$ \& $q^{d-1}\leq k_{10}-k_{00}<q^d$): By taking $a_{r-2-d} \not=0$ and all other coefficients zero, and $s=(s_0,s_1,s_2)$, the $k_{00}$-projection of $u_{s,w}$ differs from its $k_{10}$-projection, thus $\pi_{k_0}(u_s)\not= \pi_{k_1}(u_s)$.
\end{itemize}
\end{proof}

\begin{lemmaa} \label{lm.II00}
  Upon the condition underlying the statement {\bf II.0.0} in Plan~\ref{tb.pl0}, the claim~(\ref{eq.pp01}) holds.
\end{lemmaa}

\begin{proof}
There exists $\theta\in T_B$ such that $\tr_{B/A}(\theta (x_j-y_j))\in p^{r-1}B-\{0\}$. We express in their $p$-adic forms $\tr_{B/A}(\theta  x_j)$ and $\tr_{B/A}(\theta  y_j)$, namely
\begin{equation}
\tr_{B/A}(\theta  x_j) = \sum_{k=0}^{r-1} a_kp^k\ \ , \ \ \tr_{B/A}(\theta  y_j) = \sum_{k=0}^{r-1} b_kp^k.  \label{eq.pftxy}
\end{equation}
Thus
$$\sum_{k=0}^{r-1} (a_k-b_k)p^k = (a_0-b_0) + \sum_{k=1}^{r-1} (a_k-b_k)p^k\in p^{r-1}A-\{0\}$$ and  $a_0-b_0=0$.
Also
$$\sum_{k=1}^{r-1} (a_k-b_k)p^{k-1} = (a_1-b_1) + \sum_{k=2}^{r-1} (a_k-b_k)p^k\in p^{r-2}A-\{0\}$$ and  $a_1-b_1=0$.  Successively, continuing with this procedure, $\forall k\leq r-2$, $a_k=b_k$, and $(a_{r-1}-b_{r-1})p\in pA-\{0\}$. Hence $a_{r-1}\not=b_{r-1},$ and
$\Phi\left(\tr_{B/A}(\theta  x_j)\right)\not=\Phi\left(\tr_{B/A}(\theta  y_j)\right).$

Let $s_0=0$, $s_1=\theta  e_j$, $s_2=0$ and $s=(s_0,s_1,s_2)\in S$. Then, according to~(\ref{eq.rs051}),
\begin{eqnarray*}
\Phi\left(v_{s,w}(x)\right) &=& \Phi\left(\tr_{B/A}(s_0\,f(x) + s_1\cdot x) + s_2 + w\right) \\
 &=& \Phi\left(\tr_{B/A}(\theta  x_j)\right) + \Phi\left(w\right) \\
 &\not=& \Phi\left(\tr_{B/A}(\theta  y_j)\right) + \Phi\left(w\right) \\
 &=& \Phi\left(\tr_{B/A}(s_0\,f(y) + s_1\cdot y) + s_2 + w\right) \\
 &=& \Phi\left(v_{s,w}(y)\right),
\end{eqnarray*}
and, in particular,
$\pi_{k_{00}}\circ \Phi\left(v_{s,w}(x)\right) \not= \pi_{k_{10}}\circ \Phi\left(v_{s,w}(x)\right).$
Thus,  implication~(\ref{eq.pp01}) holds under these conditions.
\end{proof}

\begin{lemmaa} \label{lm.II0100}
  Upon the condition underlying the statement {\bf II.0.1.0.0} in Plan~\ref{tb.pl0}, implication~(\ref{eq.pp01}) holds.
\end{lemmaa}

\begin{proof}
Let
$s_2 = d_{r-1}p^{r-1} + a_{r-1}p^{r-1} - \tr_{B/A}((\zeta+\theta p^t) x_j) =$\linebreak $d_{r-1}p^{r-1} + b_{r-1}p^{r-1} - \tr_{B/A}((\zeta+\theta p^t) y_j).$
Then,
\begin{eqnarray*}
\Phi\left(v_{s,w}(x)\right) &=& \Phi\left(\tr_{B/A}(s_0\,f(x) + s_1\cdot x) + s_2 + w\right) \\
 &=& \Phi\left(\tr_{B/A}(\theta f(x) + (\zeta+\theta p^t) x_j) + s_2 + w\right) \\
 &=& \Phi\left(\tr_{B/A}(\theta f(x)) + d_{r-1}p^{r-1} + a_{r-1}p^{r-1} + w\right) \\
 &=& \Phi\left(\tr_{B/A}(\theta f(x))\right) + \Phi\left(d_{r-1}p^{r-1}\right) + \Phi\left(a_{r-1}p^{r-1}\right) + \Phi\left(w\right).
\end{eqnarray*}
Thus
$$\Phi\left(v_{s,w}(y)\right) = \Phi\left(\tr_{B/A}(\theta f(y))\right) + \Phi\left(d_{r-1}p^{r-1}\right) + \Phi\left(b_{r-1}p^{r-1}\right) + \Phi\left(w\right),$$
hence $\Phi\left(v_{s,w}(x)\right)\not=\Phi\left(v_{s,w}(y)\right)$.
In particular,
$\pi_{k_{00}}\circ \Phi\left(v_{s,w}(x)\right) \not= \pi_{k_{10}}\circ \Phi\left(v_{s,w}(x)\right).$
Thus,  implication~(\ref{eq.pp01}) holds under these conditions.
\end{proof}

\begin{lemmaa} \label{lm.II0101}
 Upon the condition underlying the statement {\bf II.0.1.0.1} in Plan~\ref{tb.pl0}, implication~(\ref{eq.pp01}) holds.
\end{lemmaa}

\begin{proof}
Let $\theta\in T_B$ be as in Assertion~\ref{lm.II00} above and $(s_0, s_1,s_2) = (\theta,0,0)$. Then,
$\Phi\left(v_{s,w}(x)\right) = \Phi\left(\tr_{B/A}(\theta f(x)\right) + \Phi\left(w\right)$ and $\Phi\left(v_{s,w}(y)\right) = \Phi\left(\tr_{B/A}(\theta f(y)\right) + \Phi\left(w\right).$ 
Hence $\pi_{k_{00}}\circ \Phi\left(v_{s,w}(x)\right)\not=\pi_{k_{10}}\circ \Phi\left(v_{s,w}(y)\right)$.
\end{proof}

\begin{lemmaa} \label{lm.II01100}
 Upon the condition underlying the statement {\bf II.0.1.1.0.0} in Plan~\ref{tb.pl0}, implication~(\ref{eq.pp01}) holds.
\end{lemmaa}

\begin{proof}
Let $s_0=0$, $s_1=e_j$, $s_2=0$ and $s=(s_0,s_1,s_2)\in S$. Then as in Assertion~\ref{lm.II00} we conclude that $\pi_{k_{00}}\circ \Phi\left(v_{s,w}(x)\right)\not=\pi_{k_{10}}\circ \Phi\left(v_{s,w}(y)\right)$.
\end{proof}

\begin{lemmaa} \label{lm.II01101}
 Upon the condition underlying the statement {\bf II.0.1.1.0.1} in Plan~\ref{tb.pl0}, implication~(\ref{eq.pp01}) holds.
\end{lemmaa}

\begin{proof}
There is a pair $(s_0,s_1) = (\theta,p^t\,e_j)$ in the set $D_{\eta}$, as defined in~(\ref{eq.deta}), such that
$\Phi\left(\tr_{B/A}(\theta f(x))\right) = \Phi\left(\tr_{B/A}(\theta f(y))\right),$
since $\theta\in T_A-\{0\}$. Written in $p$-adic form
$\tr_{B/A}(p^t x_j) = \sum_{i=0}^{r-2}a_ip^i + a_{r-1}p^{r-1}$, $\tr_{B/A}(p^t y_j) = \sum_{i=0}^{r-2}a_ip^i + b_{r-1}p^{r-1}$
with $a_{r-1}\not= b_{r-1}$. An adequate selection of $s_2$ gives
\begin{eqnarray*}
\Phi\left(v_{s,w}(x)\right) &=& \Phi\left(\tr_{B/A}(\theta f(x)+ \tr_{B/A}(p^t x_j) + s_2 + w\right)  \\
 &=& \Phi\left(\tr_{B/A}(\theta f(x)+ a_{r-1}p^{r-1} + w\right)  \\
 &=& \Phi\left(\tr_{B/A}(\theta f(x)\right)+ \Phi\left(a_{r-1}p^{r-1}\right) + \Phi\left(w\right).  
\end{eqnarray*}
Similarly,
$\Phi\left(v_{s,w}(y)\right) = \Phi\left(\tr_{B/A}(\theta f(y)\right)+ \Phi\left(b_{r-1}p^{r-1}\right) + \Phi\left(w\right),$
and the right sides of the above identities are different, thus implication~(\ref{eq.pp01}) holds in this case.
\end{proof}

\begin{lemmaa} \label{lm.II0111}
 Upon the condition underlying the statement {\bf II.0.1.1.1} in Plan~\ref{tb.pl0}, the claim~(\ref{eq.pp01}) holds.
\end{lemmaa}

\begin{proof}
In this case, $\pi_{k_{00}}\circ \Phi\left(\tr_{B/A}(\eta_{(r-1)n}f(x))\right) \not= \pi_{k_{10}}\circ \Phi\left(\eta_{(r-1)n}\tr_{B/A}(f(y))\right)$
and there exists $\eta_{(r-1)n}\in T(A)-\{0\}$ such that $\eta_{(r-1)n}$ does not appear in $\eta$,  because $(r-1)(n+1)<p^m-1$. Now, we choose $s_1=0\in B^n$, $s_2=0$ and $s=(\eta_{(r-1)n},0,0)$. Then,
\begin{eqnarray*}
\pi_{k_{00}}\circ \Phi\left(v_{s,w}(x)\right) &=& \pi_{k_{00}}\circ \Phi\left(\tr_{B/A}(s_0\,f(x) + s_1\cdot x) + s_2 + w\right) \\
 &=& \pi_{k_{00}}\circ \Phi\left(\tr_{B/A}(\eta_{(r-1)n} f(x)) + w\right) \\
 &\not=& \pi_{k_{10}}\circ \Phi\left(\tr_{B/A}(\eta_{(r-1)n} f(y)) + w\right) \\
 &=& \pi_{k_{10}}\circ \Phi\left(v_{s,w}(y)\right)
\end{eqnarray*}
 and implication~(\ref{eq.pp01}) holds.
\end{proof}

\begin{lemmaa} \label{lm.II10}
 Upon the condition underlying the statement {\bf II.1.0} in Plan~\ref{tb.pl1}, implication~(\ref{eq.pp01}) holds.
\end{lemmaa}

\begin{proof}
There is a $\theta\in T_B$ such that $\tr_{B/A}(\theta(x_j-y_j))j\in p^{r-1}A-\{0\}$. By writing $\tr_{B/A}(\theta x_j)$ and $\tr_{B/A}(\theta y_j)$ in $p$-adic form as in~(\ref{eq.pftxy})
we have that, as in Assertion~\ref{lm.II00}, for any $i\leq r-2$, $a_i=b_i$ and $a_{r-1}-b_{r-1}\in p^{r-1}B-\{0\}$.
Let $(s_0,s_1,s_2) = \left(0,\theta e_j,-\sum_{i=0}^{r-2}a_ip^i\right)$. 
Then
$\Phi\left(v_{s,w}(x)\right) = \Phi\left(\tr_{B/A}(a_{r-1}p^{r-1}\right) + \Phi\left(w\right)$ and $\Phi\left(v_{s,w}(y)\right) = \Phi\left(\tr_{B/A}(b_{r-1}p^{r-1}\right) + \Phi\left(w\right).$ 
Hence $\pi_{k_{00}}\circ \Phi\left(v_{s,w}(x)\right)\not=\pi_{k_{10}}\circ \Phi\left(v_{s,w}(y)\right)$.
\end{proof}

\begin{lemmaa} \label{lm.II1100}
 Upon the condition underlying the statement {\bf II.1.1.0.0} in Plan~\ref{tb.pl1}, implication~(\ref{eq.pp01}) holds.
\end{lemmaa}

\begin{proof}
There is a $s_2$ in $\left(T(B)-(\{0\}\cup\eta)\right)\times\{0\}\times L$ such that
\begin{eqnarray*}
\Phi\left(v_{s,w}(x)\right) &=& \Phi\left(\tr_{B/A}(s_0\,f(x) + s_1\cdot x) + s_2 + w\right) \\
 &=& \Phi\left(\tr_{B/A}(\theta f(x) + (\zeta+\theta p^t) x_j) + s_2 + w\right) \\
 &=& \Phi\left(\tr_{B/A}(\theta f(x)) + \tr_{B/A}(\zeta x_j) + s_2+\tr_{B/A}(\theta p^t x_j) + w\right) \\
 &=& \Phi\left(\tr_{B/A}(\theta f(x)) + c_{r-1}p^{r-1}+a_{r-1}p^{r-1} + w\right) \\
 &=& \Phi\left(\tr_{B/A}(\theta f(x))\right) + \Phi\left(c_{r-1}p^{r-1}\right)+\Phi\left(a_{r-1}p^{r-1}\right) + \Phi\left(w\right),
\end{eqnarray*}
where we have used the $p$-adic forms displayed in Plan~\ref{tb.pl0}.

{\em Mutatis mutandis} we get,
$$\Phi\left(v_{s,w}(y)\right) = \Phi\left(\tr_{B/A}(\theta f(y))\right) + \Phi\left(c_{r-1}p^{r-1}\right)+\Phi\left(b_{r-1}p^{r-1}\right) + \Phi\left(w\right),$$
hence $\Phi\left(v_{s,w}(x)\right)\not=\Phi\left(v_{s,w}(y)\right)$.
In particular,
$\pi_{k_{00}}\circ \Phi\left(v_{s,w}(x)\right) \not= \pi_{k_{10}}\circ \Phi\left(v_{s,w}(x)\right).$ 
Thus,  implication~(\ref{eq.pp01}) holds under these conditions.
\end{proof}

\begin{lemmaa} \label{lm.II1101}
 Upon the condition underlying the statement {\bf II.1.1.0.1} in Plan~\ref{tb.pl1}, implication~(\ref{eq.pp01}) holds.
\end{lemmaa}

\begin{proof}
We may proceed as in Assertion~\ref{lm.II0101} to show that implication~(\ref{eq.pp01}) holds under these conditions.
\end{proof}

\begin{lemmaa} \label{lm.III0}
 Upon the condition underlying the statement {\bf III.0} in Plan~\ref{tb.pl},  implication~(\ref{eq.pp01}) holds.
\end{lemmaa}

\begin{proof}
For any $s=(s_0,s_1,s_2)\in S $ we have
\begin{eqnarray*}
\Phi(v_{s,w_0}(x)) - \Phi(v_{s,w_1}(x)) &=& \ \ \Phi\left(\tr_{B/A}(s_0 f(x) + s_1\cdot x) + s_2 + w_0\right) \\
 & &
 - \Phi\left(\tr_{B/A}(s_0 f(x) + s_1\cdot x) + s_2 + w_1\right) \\
 &=& \Phi\left(w_0\right)  - \Phi\left(w_1\right) \not= 0.
\end{eqnarray*}
In particular,
$ \pi_{k_{00}}\circ\Phi(v_{s,w_0}(x)) \not= \pi_{k_{00}}\circ\Phi(v_{s,w_1}(x)).$ 
Thus,  implication~(\ref{eq.pp01}) holds in this case as well.
\end{proof}

\begin{lemmaa} \label{lm.III10}
 Upon the condition underlying the statement {\bf III.1.0} in Plan~\ref{tb.pl}, the claim~(\ref{eq.pp01}) holds.
\end{lemmaa}

\begin{proof}
If, written in its $p$-adic form, $\tr_{B/A}(s_0\,f(x) + s_1\cdot x) = \sum_{i=0}^{r-1}a_ip^i$, let $s_2=-\sum_{i=0}^{r-2}a_ip^i$. As in Assertion~\ref{lm.II10}, we will have
$$\pi_{k_{00}}\circ \Phi(v_{s,w_0}(x)) \not= \pi_{k_{10}}\circ \Phi(v_{s,w_1}(x)).$$
\end{proof}

\begin{lemmaa} \label{lm.III11}
 Upon the condition underlying the statement {\bf III.1.1} in Plan~\ref{tb.pl}, implication~(\ref{eq.pp01}) holds.
\end{lemmaa}

\begin{proof}
Let $s_2=0$. We will have
$\pi_{k_{00}}\circ \Phi(v_{s,w_0}(x)) \not= \pi_{k_{10}}\circ \Phi(v_{s,w_1}(x)).$
\end{proof}

\begin{lemmaa} \label{lm.IV0}
 Upon the condition underlying the statement {\bf IV.0} in Plan~\ref{tb.pl}, implication~(\ref{eq.pp01}) holds.
\end{lemmaa}

\begin{proof}
In this case, 
$\pi_{k_{00}}\circ \Phi\left(\tr_{B/A}(\eta_{(r-1)n}f(x) )\right) = \pi_{k_{10}}\circ \Phi\left(\tr_{B/A}(\eta_{(r-1)n}f(y) )\right)$ with $\eta_{(r-1)n}\in T(A)-\{0\}$ such that $\eta_{(r-1)n}\not\in\eta,$ where $\eta$ is defined en~(\ref{eq.eta}).

If $(s_0,s_1,s_2) = (\eta_{(r-1)n},0,0)$, then $\pi_{k_{00}}\circ \Phi(v_{s,w_0}(x)) \not= \pi_{k_{10}}\circ \Phi(v_{s,w_1}(x)).$
\end{proof}

\begin{lemmaa} \label{lm.IV1}
 Upon the condition underlying the statement {\bf IV.1} in Plan~\ref{tb.pl}, implication~(\ref{eq.pp01}) holds.
\end{lemmaa}

\begin{proof}
Let $\eta\in T(A)$. Then,
$\pi_{k_{00}}\circ \Phi\left(\eta \tr_{B/A}(f(x) )\right) = \eta \pi_{k_{10}}\circ \Phi\left(\tr_{B/A}(f(x) )\right)$ and $\pi_{k_{00}}\circ \Phi\left(\eta \tr_{B/A}(f(y) )\right) = \eta \pi_{k_{10}}\circ \Phi\left(\tr_{B/A}(f(y) )\right),$ 
and if there exists $\eta\in T(A)$ such that
$$\pi_{k_{00}}\circ \Phi\left(w_0\right) + \pi_{k_{00}}\circ \Phi\left(\eta \tr_{B/A}(f(x) )\right) =\pi_{k_{10}}\circ \Phi\left(w_1\right) + \pi_{k_{10}}\circ \Phi\left(\eta \tr_{B/A}(f(y) )\right)$$ then this element $\eta$ is unique. 

Let us choose $\zeta = \left\{\zeta_k\right\}_{k=0}^{q^m -(r-1)n-2}$, as was done in relation~(\ref{eq.teta}).
Thus, either
$$\pi_{k_{00}}\circ \Phi\left(w_0\right) + \pi_{k_{00}}\circ \Phi\left(\zeta_k\tr_{B/A}(f(x) )\right) \not=\pi_{k_{10}}\circ \Phi\left(w_1\right)+ \pi_{k_{10}}\circ \Phi\left(\zeta_k\tr_{B/A}(f(y) )\right)$$
or
$$\pi_{k_{00}}\circ \Phi\left(w_0\right) + \pi_{k_{00}}\circ \Phi\left(\zeta_{k'}\tr_{B/A}(f(x) )\right) \not=\pi_{k_{10}}\circ \Phi\left(w_1\right) + \pi_{k_{10}}\circ \Phi\left(\zeta_{k'}\tr_{B/A}(f(y) )\right),$$ where $\zeta_k,\zeta_{k'}\in T(A)\cap \zeta,$ $k\neq k'.$
Let $j$ be an index witnessing the relations above and $(s_0,s_1,s_2) = (\eta_j,0,0)$. Then $\pi_{k_{00}}\circ \Phi(v_{s,w_0}(x)) \not= \pi_{k_{10}}\circ \Phi(v_{s,w_1}(x)).$
\end{proof}

\bibliographystyle{splncs}
\bibliography{lopr}

\begin{thebibliography}{1}

\bibitem{DingN04}
Ding, C., Niederreiter, H.:
\newblock Systematic authentication codes from highly nonlinear functions.
\newblock IEEE Transactions on Information Theory \textbf{50}(10) (2004)
  2421--2428

\bibitem{ku15an}
Ku-Cauich, J.C., Morales-Luna, G., Tapia-Recillas, H.:
\newblock An authentication code over {G}alois rings with optimal impersonation
  and substitution probabilities.
\newblock Cryptology ePrint Archive: Report 2015/618 \textbf{{\tt
  https://eprint.iacr.org/2015/618}} (2015)

\bibitem{Ku-CauichT13}
Ku-Cauich, J.C., Tapia-Recillas, H.:
\newblock Systematic authentication codes based on a class of bent functions
  and the {G}ray map on a {G}alois ring.
\newblock SIAM J. Discrete Math. \textbf{27}(2) (2013)  1159--1170

\bibitem{OzbudakS06}
{\"O}zbudak, F., Saygi, Z.:
\newblock Some constructions of systematic authentication codes using {G}alois
  rings.
\newblock Des. Codes Cryptography \textbf{41}(3) (2006)  343--357

\bibitem{Stinson92}
Stinson, D.R.:
\newblock Combinatorial characterizations of authentication codes.
\newblock Designs, Codes and Cryptography \textbf{2}(2) (1992)  175--187

\end{thebibliography}

\end{document}